\documentclass[a4paper,14pt,reqno]{article}
\usepackage{amsmath,amsthm,amssymb,cite}
\usepackage{graphicx,epstopdf}
\usepackage{subfigure}

\usepackage{setspace,titlesec}
\titleformat{\section}[hang]
{\large\bfseries}
{\thesection}
{1em}
\nonstopmode \numberwithin{equation}{section}

\newtheorem{theorem}{Theorem}[section]
 \newtheorem{corollary}{Corollary}[section]

\newtheorem{lemma}{Lemma}[section]

\newtheorem{remark}{Remark}[section]

\allowdisplaybreaks

\begin{document}
\thispagestyle{empty}

\renewcommand{\thefootnote}

\begin{center}
{\large\bf General Viscosity Implicit Midpoint Rule For Nonexpansive Mapping}

\vspace{.4cm}
{\bf Shuja Haider Rizvi$^*${\footnote{$^*$Corresponding author; E-mail address:shujarizvi07@gmail.com (S.H. Rizvi)}}}

\vspace{0.15in}
{\it
${}^{}$Department of Mathematics, Babu Banarasi Das University, Lucknow 226028, India} \\
\end{center}

\vspace{.7cm}
\baselineskip=18pt
%\bigskip

\noindent{\bf Abstract:} In this work, we suggest a general viscosity  implicit midpoint rule for nonexpansive mapping in the framework of Hilbert space.
Further, under the certain conditions imposed on the sequence of parameters, strong convergence theorem is  proved by the sequence generated by the proposed
iterative scheme, which, in addition, is the unique solution of the  variational inequality problem. Furthermore, we provide some applications  to variational inequalities,
Fredholm integral equations, and nonlinear evolution equations. The results  presented in this work  may be treated as an improvement, extension and
refinement of some corresponding ones in the literature.\\



\section{Introduction}

Throughout the paper unless otherwise stated, $H$ denotes a real Hilbert space, we denote the norm and inner product of  $H$  by $\|\cdot\|$, and
$\langle .,. \rangle$ respectively. Let $K$ be a nonempty, closed and convex subset of $H$. Let $\{x_{n}\}$ be any sequence in $H$, then $x_{n}\rightarrow x$
(respectively, $x_{n}\rightharpoonup x$) will denote strong (respectively, weak) convergence of the sequence $\{x_{n}\}$.

\vspace{.25cm}
A mapping $S:H\to H$ is said to be  \verb"contraction mapping" if there exists a constant $\alpha\in (0,1)$ such that
$$\|Sx-Sy\|\leq\alpha\|x-y\|,$$
for all $x,y\in H$. If $\alpha=1$ then $S:H\to H$ is said to be \verb"nonexpansive mapping" i.e., $\|Sx-Sy\|\leq\|x-y\|$, for all $x,y\in H$. We use
Fix$(S)$ to denote the set of fixed points of $S$. An operator  $B:H\to H$ is said to be \verb"strongly positive bounded linear operator",  if there exists
a constant $\bar\gamma>0$ such that
$$\langle Bx,x\rangle \geq \bar\gamma\|x\|^2,~~~~~\forall x\in H.$$

The viscosity approximation method of selecting a particular fixed point of a given nonexpansive mapping was proposed by Moudafi {\rm\cite{Moudafi}} in the
framework of a Hilbert space, which  generates the sequence $\{x_n\}$ by the following iterative scheme:
\begin{equation}\label{s1e1}
x_{n+1}=\alpha_n Q(x_n)+ (1-\alpha_n)Sx_n,~ n\geq 0,
\end{equation}
where $\{\alpha_n\}\subset [0,1]$ and $Q$ is a contraction mapping on $H$. Note that the iterative scheme (\ref{s1e1})  generalize the results of
Browder {\rm\cite{Browder1}} and Halpern \cite{Halpern} in another direction.  The convergence of the explicit iterative scheme (\ref{s1e1}) has been the
subject of many authors because under suitable conditions these iteration converge strongly to the unique solution $q\in {\rm Fix}(S)$ of the variational
inequality
\begin{equation}\label{s1e2}
\langle(I-Q)q,x-q\rangle \geq 0,~~ \forall x\in {\rm Fix}(S).
\end{equation}
This fact allows us to apply this method to convex optimization, linear programming and monotone inclusions. In 2004, Xu \cite{Xu} extended the result of
Moudafi {\rm\cite{Moudafi}} to uniformly smooth Banach spaces and obtained strong convergence theorem. For related work, see \cite{Chidume1,Berinde,Yamada}.

\vspace{.25cm}
In 2006, Marino and Xu {\rm\cite{Marino}}  introduced  the following iterative scheme based on viscosity approximation method,  for fixed point problem
 for a nonexpansive mapping $S$ on $H$:
\begin{equation}\label{s1e3}
x_{n+1}=\alpha_n \gamma Q(x_n)+(I-\alpha_nB)Sx_n,~ n\geq 0,
\end{equation}
where $Q$ is a contraction mapping on $H$ with constant $\alpha >0$, $B$ is a strongly positive  self-adjoint bounded linear operator on $H$ with constant
$\bar{\gamma}>0$ and $\gamma \in (0, \frac{\bar{\gamma}}{\alpha})$. They proved that the sequence $\{x_n\}$ generated by (\ref{s1e3}) converge strongly to the
unique solution of the variational inequality
\begin{equation}\label{s1e4}
\langle (B-\gamma Q)z,x-z\rangle\geq0,~~~\forall x\in {\rm Fix}(S),
\end{equation}
which is the optimality condition for the minimization problem
$$\min\limits_{x\in {\rm Fix}(S)}\dfrac{1}{2}\langle Bx,x\rangle-h(x),$$
where $h$ is the potential function for $\gamma Q$.

\vspace{.25cm}
The implicit midpoint rule is one of the powerful numerical methods for solving ordinary differential equations and differential algebraic equations.
For related works, we refer to {\rm\cite{Hofer,Bader,Deuflhard,Auzinger,Bayreuth,Schneider,Somalia,Somalia1}} and the references cited therein. For instance,
consider the initial value problem for the differential equation  $y'(t)=f(y(t))$ with the initial condition $y(0)=y_0$, where $f$ is a continuous function
from $R^d$ to $R^d$. The implicit midpoint rule in which generates a sequence $\{y_n\}$ by the following the recurrence relation
$$\dfrac{1}{h}(y_{n+1}-y_{n})=f\left(\dfrac{y_{n+1}-y_{n}}{2}\right).$$

In 2014,  implicit midpoint rule has been extended by Alghamdi {\it et al.} {\rm\cite{Alghamdi}} to nonexpansive mappings, which generates a sequence
$\{x_n\}$ by the following implicit iterative scheme:
\begin{equation}\label{s1e5}
x_{n+1}=\alpha_{n} x_n+(1-\alpha_{n})S\left(\dfrac{x_{n}+x_{n+1}}{2}\right),~ n\geq 0,
\end{equation}
Recently, Xu {\it et al.} {\rm\cite{Xu1}} extended and generalized the results of Alghamdi {\it et al.} {\rm\cite{Alghamdi}} and presented the following
viscosity implicit midpoint rule for nonexpansive mapping, which generates a sequence $\{x_n\}$ by the following implicit iterative scheme:
\begin{equation}\label{s1e6}
x_{n+1}=\alpha_{n}Q(x_n)+(1-\alpha_{n})S\left(\dfrac{x_{n}+x_{n+1}}{2}\right),~ n\geq 0,
\end{equation}
where $\{\alpha_n\}\subset [0,1]$ and $S$ is a nonexpansive mapping. They proved that under some mild conditions, the sequence generated by  (\ref{s1e6})
converge in norm to fixed point of nonexpansive mapping, which, in addition, solves the variational inequality (\ref{s1e2}). Further related work, see {\rm\cite{Yao,Zhao}}.

\vspace{.25cm}
Motivated by the work of  Moudafi {\rm\cite{Moudafi}}, Xu \cite{Xu}, Marino and Xu {\rm\cite{Marino}},   Alghamdi {\it et al.} {\rm\cite{Alghamdi}} and
Xu {\it et al.} {\rm\cite{Xu1}}, and by the ongoing research in this direction,  we suggest and analyze general viscosity  implicit midpoint   iterative
scheme for fixed point of  nonexpansive  mapping in  real Hilbert space. Further, based on these general viscosity  implicit midpoint   iterative  scheme,
we prove the strong convergence theorems for a nonexpansive  mapping.  Furthermore, some consequences from these theorems are also derived.
The results and methods presented here extend and  generalize the corresponding results and methods given in \cite{Moudafi,Xu,Marino,Alghamdi,Xu1}.

\section{Preliminaries}

We recall some concepts and results which are needed in sequel.

\vspace{.25cm}
For every point  $x \in H$, there exists a unique nearest point in $K$ denoted by $P_ {K} x$ such that
\begin{equation}\label{c1s2e3*}
\|x-P_{K}x\|\leq \|x-y\|, ~~ \forall  y \in K.
\end{equation}
\begin{remark}{\rm\cite{Bauschke}}   It is well known that $P_{K}$ is nonexpansive mapping and satisfies
\begin{equation}\label{c1s2e4}
\langle x-y ,P_{K}x-P_{K}y \rangle \geq \|P_{K}x-P_{K}y\|^2, ~~\forall x,y \in H.\
\end{equation}
Moreover, $P_{K}x$ is characterized by the fact $P_{K}x\in K$ and
\begin{equation}\label{c1s2e5}
\langle x-P_{K}x,y-P_{K}x \rangle \leq 0.
\end{equation}
\end{remark}

\vspace{.25cm}
The following Lemma is the well known demiclosedness principles for nonexpansive mappings.

\begin{lemma}\label{l1}{\rm\cite{Bauschke,Goebel}}
Assume that $S$ be a nonexpansive self mapping of a closed and  convex subset $K$ of a Hilbert space $H$. If $S$ has a fixed point, then $I-S$ is demiclosed,
i.e., whenever $\{x_n\}$ is a sequence in $K$ converging weakly to some $x\in K$ and the sequence $\{(I-S)x_n\}$ converges strongly to some $y$, it follows
that $(I-S)x=y$.
\end{lemma}

\begin{lemma}\label{l2}{\rm\cite{Bauschke,Goebel}} In real Hilbert space $H$, the following hold:
\begin{enumerate}
\item[{(i)}]\begin{equation}\label{c1s2e10}
\|x+y\|^2 \leq \|x\|^2+2\langle y,x+y \rangle, ~~\forall x,y \in H;
\end{equation}
\item[{(ii)}]\begin{equation}\label{c1s2e11} \|\lambda x+(1-\lambda)y\|^2=\lambda\|x\|^2+(1-\lambda)\|y\|^2-\lambda(1-\lambda)\|x-y\|^2,
\end{equation} for all $x,y\in H$ and $\lambda\in (0,1)$.
\end{enumerate}
\end{lemma}

\begin{lemma}\label{l3}{\rm\cite{Marino}} Assume that $B$ is a strongly positive  self-adjoint bounded linear operator on a Hilbert space $H$ with constant $\bar\gamma>0$ and $0<\rho\leq\|B\|^{-1}$. Then $\|I-\rho B\|\leq1-\rho\bar\gamma$.
\end{lemma}

\begin{lemma}\label{l4}{\rm\cite{Xu}}. Let  $\{a_{n}\}$ be a sequence of nonnegative real numbers such that
$$a_{n+1}\leq (1-\beta_{n})a_{n}+\delta_{n},~~~n \geq 0,$$
where $\{\beta_{n}\}$ is a sequence in $(0,1)$ and $\{\delta_{n}\}$ is a sequence in ${\mathbb R}$ such that
\begin{enumerate}
\item[{(i)}] $\sum\limits_{n=1}^{\infty} \beta_{n}=\infty;$
\item[{(ii)}] $\lim \sup\limits_{n \to \infty}\dfrac{\delta_{n}}{\beta_{n}}\leq 0$~ or~ $\sum\limits_{n=1}^{\infty} \vert\delta_{n}\vert <\infty$.
\end{enumerate}
Then $\lim\limits_{n \to\infty} a_{n}=0.$
\end{lemma}

\section{General Viscosity Implicit Midpoint Rule}

In this section, we prove a strong convergence theorem based on the general viscosity implicit midpoint rule for fixed point of nonexpansive mapping.

\begin{theorem}\label{T1} Let $H$ be a real Hilbert space and $B:H\to H$  be a strongly positive bounded linear operator with constant $\bar\gamma>0$ such that $0<\gamma<\frac{\bar\gamma}{\alpha}<\gamma+\frac{1}{\alpha}$ and $Q:H\to H$ be a contraction  mapping with constant $\alpha\in (0,1)$. Let $S:H\to H$  be a nonexpansive mapping such that ${\rm Fix}(S)\neq\emptyset$. Let the iterative sequence $\{x_{n}\}$ be  generated by the following general viscosity implicit midpoint iterative schemes:
\begin{equation}\label{s3e1}
x_{n+1}=\alpha_{n}\gamma Q(x_n)+(1-\alpha_{n}B)S\left(\dfrac{x_{n}+x_{n+1}}{2}\right),~ n\geq 0,
\end{equation}
where $\{\alpha_{n}\}$ is the sequence in $(0,1)$ and satisfying the following conditions
\begin{enumerate}
\item[{(i)}]  $\lim\limits_{n \to \infty} \alpha_{n}=0$;
\item[{(ii)}] $\sum\limits_{n=0}^{\infty} \alpha_{n}=\infty$;
\item[{(iii)}] $\sum\limits_{n=1}^{\infty} \vert \alpha_{n}-\alpha_{n-1}\vert <\infty$ or $\lim\limits_{n\to\infty}\dfrac{\alpha_{n+1}}{\alpha_n}=1$.
\end{enumerate}
Then the sequence $\{x_{n}\}$  converge strongly to $z \in {\rm Fix}(S)$, where $z=P_{{\rm Fix}(S)}Q(z)$. In other words, which is also unique solution of variational inequality {\rm (\ref{s1e4})}.
\end{theorem}

\begin{proof} Note that from condition (i), we may assume without loss of generality that $\alpha_n\leq(1-\beta_n)\|B\|^{-1}$ for all $n$. From Lemma \ref{l3}, we know that if $0<\rho\leq\|B\|^{-1}$, then $\|I-\rho B\|\leq 1-\rho\bar\gamma$. We will assume that $\|I-B\|\leq1-\bar\gamma$.

\vspace{.25cm}
Since $B$ is strongly positive self-adjoint bounded linear operator on $H$, then
$$\|B\|=\sup\{\vert\langle Bu,u\rangle\vert:u\in H, \|u\|=1\}.$$
Observe that
\begin{eqnarray}\nonumber
\langle (I-\alpha_n B)u,u\rangle&=&1-\alpha_n\langle Bu,u\rangle\\\nonumber
&\geq&1-\alpha_n\|B\| ~\geq0,
\end{eqnarray}
which implies that $(1-\alpha_n B)$ is positive. It follows that
\begin{eqnarray}\nonumber
\|(I-\alpha_n B\|&=&\sup\{\langle ((1-\alpha_n B)u,u\rangle:u\in H, \|u\|=1\}\\\nonumber
&=&\sup\{1-\alpha_n \langle Bu,u\rangle:u\in H, \|u\|=1\}\\\nonumber
&\leq&1-\alpha_n\bar\gamma.
\end{eqnarray}
Let $q=P_{{\rm Fix}(S)}.$ Since $Q$ is a contraction mapping with constant $\alpha\in (0,1)$. It follows that
\begin{eqnarray}\nonumber
\|q(I-B+\gamma Q)(x)-q(I-B+\gamma Q)(y)\|&\leq&\|(I-B+\gamma Q)(x)-(I-B+\gamma Q)(x)\|\\\nonumber
&\leq&\|I-B\|\|x-y\|+\gamma\|Q(x)-Q(y)\|\\\nonumber
&\leq& (1-\bar\gamma)\|x-y\|+\gamma\alpha\|x-y\|\\\nonumber
&\leq&(1-(\bar\gamma-\gamma\alpha))\|x-y\|,
\end{eqnarray}
for all $x,y\in H$. Therefore, the mapping  $q(I-B+\gamma Q)$ is a contraction mapping from $H$ into itself. It follows from Banach contraction principle
that there exists an element $z\in H$ such that $z=q(I-B+\gamma Q)z=P_{{\rm Fix}(S)}(I-B+\gamma Q)(z)$.

\vspace{.25cm}
Let $p\in {\rm Fix}(S)$, we compute
\begin{eqnarray}\nonumber
\|x_{n+1}-p\|&=&\Big\|\alpha_{n}\gamma Q(x_n)+(1-\alpha_n B)S\left(\dfrac{x_{n}+x_{n+1}}{2}\right)-p\Big\|\\\nonumber
&\leq& \alpha_{n}\|\gamma Q(x_n)-Bp\|+(1-\alpha_n \bar\gamma)\left\|S\left(\dfrac{x_{n}+x_{n+1}}{2}\right)-p\right\|\\\nonumber
&\leq& \alpha_{n}\left[\gamma\|Q(x_n)-Q(p)\|+\|\gamma Q(p)-Bp\|\right]+(1-\alpha_n \bar\gamma)\left\|\left(\dfrac{x_{n}+x_{n+1}}{2}\right)-p\right\|\\\nonumber
&\leq& \alpha_{n}\gamma\alpha\|x_n-p\|+\alpha_{n}\|\gamma Q(p)-Bp\|+\dfrac{(1-\alpha_n \bar\gamma)}{2}\left(\|x_{n}-p\|+\|x_{n+1}-p\|\right),\nonumber
\end{eqnarray}
which implies that
\begin{eqnarray}\nonumber
\dfrac{(1+\alpha_n \bar\gamma)}{2}\|x_{n+1}-p\|&\leq& \left[\alpha_n\gamma\alpha+\dfrac{(1-\alpha\bar\gamma)}{2}\right]\|x_n-p\|+\alpha_{n}\|\gamma Q(p)-Bp\|\\\nonumber
\|x_{n+1}-p\|&\leq& \left[\dfrac{1+2(\gamma\alpha-\bar\gamma)\alpha_n}{1+\alpha_n\bar\gamma}\right]\|x_n-p\|+\dfrac{2\alpha_{n}}{1+\alpha_n\bar\gamma}\|\gamma Q(p)-Bp\|\\\nonumber
&\leq& \left[1-\dfrac{2(\bar\gamma-\gamma\alpha)\alpha_n}{1+\alpha_n\bar\gamma}\right]\|x_n-p\|+\dfrac{2\alpha_{n}}{1+\alpha_n\bar\gamma}\|\gamma Q(p)-Bp\|\\\nonumber
&\leq& \left[1-\dfrac{2(\bar\gamma-\gamma\alpha)\alpha_n}{1+\alpha_n\bar\gamma}\right]\|x_n-p\|+\dfrac{2\alpha_{n}(\bar\gamma-\gamma\alpha)}{1+\alpha_n\bar\gamma}
\dfrac{\|\gamma Q(p)-Bp\|}{(\bar\gamma-\gamma\alpha)}.\nonumber
\end{eqnarray}
Consequently, we get
\begin{equation}
\|x_{n+1}-p\|\leq\max \Big\{\|x_{n}-p\|,\dfrac{\|\gamma Q(p)-Bp\|}{\bar\gamma-\gamma\alpha}\Big\}.\nonumber
\end{equation}
Therefore by using induction, we obtain
\begin{equation}\label{s3e2}
\|x_{n+1}-p\| \leq  \max \Big\{\|x_{0}-p\|,\dfrac{\|\gamma Q(p)-Bp\|}{\bar\gamma-\gamma\alpha}\Big\}.
\end{equation}

\vspace{.25cm}
Hence  the sequence $\{x_{n}\}$ is bounded.

\vspace{.25cm}
Next, we show that the sequence $\{x_{n}\}$ is asymptotically regular, i.e., $\lim\limits_{n\to\infty}\|x_{n+1}-x_n\|=0$. It follows from (\ref{s3e1}) that
\begin{eqnarray}\label{s3e3}\nonumber
\|x_{n+1}-x_n\|&=&\Big\|\alpha_{n}\gamma Q(x_n)+(1-\alpha_n B)S\left(\dfrac{x_{n}+x_{n+1}}{2}\right)\\\nonumber
&&-\Big[\alpha_{n-1}\gamma Q(x_{n-1})+(1-\alpha_{n-1}B)\left(\dfrac{x_{n-1}+x_{n}}{2}\right)\Big]\Big\|\\\nonumber
&=&\Big\|(1-\alpha_n B)\left[S\left(\dfrac{x_{n}+x_{n+1}}{2}\right)-S\left(\dfrac{x_{n-1}+x_{n}}{2}\right)\right]\\\nonumber
&&+(\alpha_{n-1}B-\alpha_{n}B)\left[S\left(\dfrac{x_{n-1}+x_{n}}{2}\right)-\gamma Q(x_{n-1})\right]+\alpha_{n}(\gamma Q(x_n)-\gamma Q(x_{n-1}))\Big\|\\\nonumber
&\leq&(1-\alpha_n\bar\gamma)\Big\|S\left(\dfrac{x_{n}+x_{n+1}}{2}\right)-S\left(\dfrac{x_{n-1}+x_{n}}{2}\right)\Big\|+M\vert\alpha_{n-1}-\alpha_{n}\vert\\\nonumber
&&+\alpha_{n}\gamma\|Q(x_{n})-Q(x_{n-1})\|\\\nonumber
&\leq&\dfrac{(1-\alpha_n\bar\gamma)}{2}\Big[\|x_{n+1}-x_{n}\|+\|x_{n}-x_{n-1}\|\Big]+M\vert\alpha_{n-1}-\alpha_{n}\vert+\alpha_{n}\gamma\alpha\|x_{n}-x_{n-1}\|,\nonumber
\end{eqnarray}
where $M:=\sup \left\{S\left(\dfrac{x_{n}+x_{n+1}}{2}\right)+\gamma\|Q(x_n)\|:n\in\mathbb{N}\right\}$. It follows that
\begin{eqnarray}\label{s3e3}\nonumber
\dfrac{(1+\alpha_n\bar\gamma)}{2}\|x_{n+1}-x_n\|&\leq&\dfrac{(1-\alpha_n\bar\gamma)}{2}\|x_{n}-x_{n-1}\|+M\vert\alpha_{n-1}-\alpha_{n}\vert+
\alpha_{n}\gamma\alpha\|x_{n}-x_{n-1}\|\\\nonumber
\|x_{n+1}-x_{n}\|&\leq&\dfrac{1+2(\gamma\alpha-\bar\gamma)\alpha_{n}}{1+\alpha_n \bar\gamma}\|x_n-x_{n-1}\|+\dfrac{2M}{1+\alpha_n\bar\gamma}\vert\alpha_{n-1}-\alpha_{n}\vert\\\nonumber
&\leq&\left(1-\dfrac{2(\gamma\alpha-\bar\gamma)\alpha_{n}}{1+\alpha_n\bar\gamma}\right)\|x_n-x_{n-1}\|
+\dfrac{2M}{1+\alpha_n\bar\gamma}\vert\alpha_{n-1}-\alpha_{n}\vert.\nonumber
\end{eqnarray}
By using the conditions (i)-(iii) of Lemma \ref{l4}, we obtain

\begin{equation}\label{s3e15}
\lim\limits_{n\to\infty}\|x_{n+1}-x_{n}\|=0.
\end{equation}

Next, we show that
\begin{equation}\nonumber
\lim\limits_{n\to\infty}\|x_n-Sx_n\|=0.
\end{equation}
We can write
\begin{eqnarray}\label{s3e4}\nonumber
\|x_n-Sx_n\|&\leq&\|x_{n}-x_{n+1}\|+\Big\|x_{n+1}-S\left(\dfrac{x_{n}+x_{n+1}}{2}\right)\Big\|\\\nonumber
&&+\Big\|S\left(\dfrac{x_{n}+x_{n+1}}{2}\right)-Sx_{n}\Big\|\\\nonumber
&\leq&\|x_{n}-x_{n+1}\|+\alpha_{n}\Big\|\gamma Q(x_{n})+(1-\gamma B)S\left(\dfrac{x_{n}+x_{n+1}}{2}\right)\Big\|+\dfrac{1}{2}\|x_{n+1}-x_{n}\|\\\nonumber
&\leq&\dfrac{3}{2}\|x_{n+1}-x_{n}\|+\alpha_{n}\Big\|\gamma Q(x_{n})-S\left(\dfrac{x_{n}+x_{n+1}}{2}\right)\Big\|\\\nonumber
&\leq&\dfrac{3}{2}\|x_{n+1}-x_{n}\|+\alpha_{n}M.
\end{eqnarray}
It follows from condition (i) and (\ref{s3e3}), we obtain
\begin{equation}\nonumber
\lim\limits_{n\to\infty}\|x_n-Sx_n\|=0.
\end{equation}

Since $\{x_{n}\}$ is bounded,   there exists a subsequence $\{x_{n_{k}}\}$ of $\{x_{n}\}$ such that $x_{n_{k}}\rightharpoonup \hat x$ say. Next, we claim
that $\lim\sup\limits _{n \to \infty} \langle Q(z)-z,x_{n}-z\rangle \leq 0$, where $z=P_{{\rm Fix}(S)}(I-B+\gamma Q)z$.
To show this inequality, we consider a  subsequence $\{x_{n_{k}}\}$ of $\{x_{n}\}$ such that $x_{n_{k}}\rightharpoonup \hat x$,
\begin{eqnarray}\label{s3e4}\nonumber
\lim\sup\limits_{n \to \infty} \langle (B-\gamma Q)z-z,x_{n}-z\rangle &=&\limsup\limits_{n \to \infty} \langle (B-\gamma Q)z-z,x_{n}-z\rangle\\\nonumber
&=&\lim\sup\limits_{k \to \infty} \langle (B-\gamma Q)z-z,x_{n_{k}}-z\rangle\\
&=&\langle (B-\gamma Q)z-z,\hat x-z\rangle ~\leq~ 0.
\end{eqnarray}
Finally, we show that $x_{n}\to z$. It follows from Lemma \ref{l2}  that
\begin{eqnarray}\nonumber
\|x_{n+1}-z\|^2 &=&\Big\|\alpha_n \gamma Q(x_n)+(I-\alpha_n B)S\left(\dfrac{x_{n}+x_{n+1}}{2}\right)-z\Big\|^2\\\nonumber
&=&\Big\|\alpha_n (\gamma Q(x_n)-Bz)+(I-\alpha_n B)S\left(\dfrac{x_{n}+x_{n+1}}{2}\right)-z\Big\|^2\\\nonumber
&\leq&\Big\|(I-\alpha_n B) S\left(\dfrac{x_{n}+x_{n+1}}{2}\right)-z\Big\|^2+2\alpha_n\langle \gamma Q(x_n)-Bz,x_{n+1}-z\rangle\\\nonumber
&\leq&(1-\alpha_n\bar\gamma)^2\Big\|S\left(\dfrac{x_{n}+x_{n+1}}{2}\right)-z\Big\|^2+2\alpha_n\gamma\|Q(x_n)-Q(z)\|\|x_{n+1}-z\|\\\nonumber
&&+2\alpha_n\langle \gamma Q(z)-Bz,x_{n+1}-z\rangle\\\nonumber
&\leq&(1-\alpha_n\bar\gamma)^2\Big\|\dfrac{x_{n}+x_{n+1}}{2}-z\Big\|^2+2\alpha_n\gamma\alpha\|x_n-z\|\|x_{n+1}-z\|\\\nonumber
&&+2\alpha_n\langle \gamma Q(z)-Bz,x_{n+1}-z\rangle\\\nonumber
&\leq&(1-\alpha_n\bar\gamma)^2\|x_n-z\|^2+2\alpha_n\gamma\alpha\|x_n-z\|\|x_{n+1}-z\|\\\nonumber
&&+2\alpha_n\langle \gamma Q(z)-Bz,x_{n+1}-z\rangle\\\nonumber
&\leq&(1-\alpha_n\bar\gamma)^2\Big[\dfrac{1}{2}\|x_{n}-z\|^2+\dfrac{1}{2}\|x_{n+1}-z\|^2-\dfrac{1}{4}\|x_{n+1}-x_{n}\|^2\Big]\\\nonumber
&&+2\alpha_n\gamma\alpha\Big[\|x_n-z\|^2+\|x_{n+1}-z\|^2\Big]+2\alpha_n\langle \gamma Q(z)-Bz,x_{n+1}-z\rangle\\\nonumber
&\leq&\Big[\dfrac{(1-\alpha_n\bar\gamma)^2}{2}+\alpha_n\gamma\alpha\Big](\|x_n-z\|^2+\|x_{n+1}-z\|^2)\\\nonumber
&&+2\alpha_n\langle \gamma Q(z)-Bz,x_{n+1}-z\rangle\\\nonumber
&\leq&\dfrac{1-2\alpha_n\bar\gamma+2\alpha_{n}\gamma\alpha}{2}(\|x_n-z\|^2+\|x_{n+1}-z\|^2)+\alpha_{n}^2\bar\gamma^2M_1\\\nonumber
&&+2\alpha_n\langle \gamma Q(z)-Bz,x_{n+1}-z\rangle.
\end{eqnarray}
This implies that
\begin{eqnarray}\label{s3e5}\nonumber
\|x_{n+1}-z\|^2&\leq& \dfrac{1-2(\bar\gamma-\gamma\alpha)\alpha_{n}}{1+2(\bar\gamma-\gamma\alpha)\alpha_{n}}\|x_{n}-z\|^2+\dfrac{2\alpha_{n}\bar\gamma^2}{1+2(\bar\gamma-\gamma\alpha)\alpha_{n}}
M_1\\\nonumber
&&+\dfrac{4\alpha_{n}}{1+2(\bar\gamma-\gamma\alpha)\alpha_{n}}\langle \gamma Q(z)-Bz,x_{n+1}-z\rangle\\\nonumber
&=&\Big[1-\dfrac{4(\bar\gamma-\gamma\alpha)\alpha_{n}}{1+2(\bar\gamma-\gamma\alpha)\alpha_{n}}\Big]\|x_{n}-z\|^2+\dfrac{2\alpha_{n}\bar\gamma^2}{1+2(\bar\gamma-\gamma\alpha)\alpha_{n}}
M_1\\\nonumber
&&+\dfrac{4\alpha_{n}}{1+2(\bar\gamma-\gamma\alpha)\alpha_{n}}\langle \gamma Q(z)-Bz,x_{n+1}-z\rangle\\
&=&(1-\delta_n)\|x_{n}-z\|^2+\delta_n \sigma_n,
\end{eqnarray}
where $M_1:=\sup\{\|x_{n}-z\|^2:n\geq1\}$, \\
$\delta_n=\dfrac{4(\bar\gamma-\gamma\alpha)\alpha_{n}}{1+2(\bar\gamma-\gamma\alpha)\alpha_n}$ and
$\sigma_{n}=\dfrac{(\alpha_{n}\bar\gamma^2)M_1}{1+2(\bar\gamma-\gamma\alpha)\alpha_n)}+\dfrac{4\alpha_n}{1+2(\bar\gamma-\gamma\alpha)\alpha_n}\langle \gamma Q(z)-Bz,x_{n+1}-z\rangle.$
Since  $\lim\limits_{n \to \infty} \alpha_{n}=0$ and $\sum\limits_{n=0}^{\infty} \alpha_{n}=\infty$, it is easy to see that  $\lim\limits_{n \to \infty} \delta_{n}=0$, $\sum\limits_{n=0}^{\infty} \delta_{n}=\infty$  and $\lim \sup\limits_{n \to \infty}{\sigma_{n}}\leq 0$. Hence from  (\ref{s3e4}), (\ref{s3e5}) and   Lemma \ref{l4}, we deduce that $x_{n} \to z$. This completes the proof.
\end{proof}

As a direct consequences of Theorem {\rm\ref{T1}}, we obtain the following result  due to Xu {\it et al.} {\rm\cite{Xu1}} for fixed point of nonexpansive
mapping.  Take $\gamma:=1$ and $B:=I$ in Theorem {\rm \ref{T1}} then the following Corollary  is obtained.
\begin{corollary}{\rm\cite{Xu1}}
Let $H$ be a real Hilbert space and $Q:H\to H$ be a contraction  mapping with constant $\alpha\in (0,1)$. Let $S:H\to H$  be a nonexpansive mapping such that
${\rm Fix}(S)\neq\emptyset$. Let the iterative sequence $\{x_{n}\}$ be  generated by the following general viscosity implicit midpoint iterative schemes:
\begin{equation}\label{s3e6}
x_{n+1}=\alpha_{n}Q(x_n)+(1-\alpha_{n})S\left(\dfrac{x_{n}+x_{n+1}}{2}\right),~ n\geq 0,
\end{equation}
where $\{\alpha_{n}\}$ is the sequence in $(0,1)$ and satisfying the conditions (i)-(iii) of Theorem {\rm \ref{T1}}. Then the sequence $\{x_{n}\}$  converge
strongly to $z \in {\rm Fix}(S)$, which, in addition also solves variational inequality {\rm (\ref{s1e2})}.
\end{corollary}

The following Corollary is due to Alghamdi {\it et al.} {\rm\cite{Alghamdi}} for fixed point problem of nonexpansive mapping. Take $\gamma:=1$ and $Q,B:=I$ in
Theorem \ref{T1} then the following Corollary is obtained.

\begin{corollary}{\rm\cite{Alghamdi}}
Let $H$ be a real Hilbert space and $Q:H\to H$ be a contraction  mapping with constant $\alpha\in (0,1)$. Let $S:H\to H$  be a nonexpansive mapping such that
${\rm Fix}(S)\neq\emptyset$. Let the iterative sequence $\{x_{n}\}$ be  generated by the following general viscosity implicit midpoint iterative schemes:
\begin{equation}\label{s3e7}
x_{n+1}=\alpha_{n}x_n+(1-\alpha_{n})S\left(\dfrac{x_{n}+x_{n+1}}{2}\right),~ n\geq 0,
\end{equation}
where $\{\alpha_{n}\}$ is the sequence in $(0,1)$ and satisfying the conditions (i)-(iii) of Theorem {\rm \ref{T1}}. Then the sequence $\{x_{n}\}$  converge strongly to $z \in {\rm Fix}(S)$.
\end{corollary}

\begin{remark}
Theorem {\rm\ref{T1}} extends and generalize the viscosity implicit midpoint rule of Xu {\it et al.} {\rm \cite{Xu}} and the implicit midpoint rule of
Alghamdi {\it et al.} {\rm\cite{Alghamdi}} to a general viscosity implicit midpoint rule for a nonexpansive mappings, which also includes the results of
{\rm \cite{Moudafi,Marino}} as special cases.
\end{remark}

\section{Applications}

\subsection{Application to Variational Inequalities}

We consider the following classical \verb"variational inequality problem"  (In short, VIP): Find $x^* \in K$ such that
\begin{equation}\label{s4e1}
\langle Ax^*,x-x^*\rangle\geq0,~~~ \forall  x\in K,
\end{equation}
where $A$ is a single-valued monotone mapping on $H$ and $K$ is a closed and convex subset of $H$. We assume $K\subset dom(A)$. An example of VIP (\ref{s4e1}) is
the \verb"constrained minimization problem" :  Find $x^* \in K$ such that
\begin{equation}\label{s4e2}
\min\limits_{x\in K}\phi(x^*)
\end{equation}
where $\phi:H\to \mathbb{R}$ is a lower-semicontinuous convex function. If $\phi$ is (Frechet) differentiable, then the minimization problem (\ref{s4e2}) is
equivalently reformulated as VIP (\ref{s4e1}) with $A=\nabla\phi$. Notice that the VIP (\ref{s4e1}) is equivalent to the following fixed point problem, for any $\lambda>0$,
\begin{equation}\label{s4e3}
Sx^*=x^*,~~~~~~~Sx:=P_{K}(I-\lambda A)x.
\end{equation}
If $A$ is Lipschitz continuous and strongly monotone, then, for $\lambda>0$ small enough, $S$ is a contraction mapping and its unique fixed point is also the
unique solution of the VIP (\ref{s4e1}). However, if $A$ is not strongly monotone, $S$ is no longer a contraction, in general. In this case we must deal with
nonexpansive mappings for solving the VIP (\ref{s4e1}). More precisely, we assume
\begin{enumerate}
\item[{(i)}] $A$ is $\theta$-\verb"Lipschitz continuous" for some $\theta>0$, i.e.,
$$\|Ax-Ay\|\leq \theta\|x-y\|,~~\forall x,y\in H.$$
\item[{(ii)}] $A$ is $\mu$-\verb"inverse strongly monotone" ($\mu$-ism) for some $\mu>0$, namely,
$$\langle Ax-Ay,x-y\geq \mu\|Ax-Ay\|,~~\forall x,y\in H.$$
\end{enumerate}
It is well known that by using the  conditions (i) and (ii),  the operator $S=P_{K}(I-\lambda A)$ is nonexpansive provided that $0<\lambda<2\mu$. It turns out that
for this range of values of $\lambda$, fixed point algorithms can be applied to solve the VIP (\ref{s4e1}). Applying Theorem \ref{T1},  we get the
following result.

\begin{theorem}\label{T2} Assume that {\rm VIP (\ref{s4e1})} is solvable in which $A$ satisfies the conditions (i) and (ii) with $0<\lambda<2\mu$. Let
$B:H\to H$  be a strongly positive bounded linear operator with constant $\bar\gamma>0$ such that
$0<\gamma<\frac{\bar\gamma}{\alpha}<\gamma+\frac{1}{\alpha}$ and $Q:H\to H$ be a contraction  mapping with constant $\alpha\in (0,1)$.
Let the iterative sequence $\{x_{n}\}$ be  generated by the following general viscosity implicit midpoint iterative schemes:
\begin{equation}\label{s4e4}
x_{n+1}=\alpha_{n}\gamma Q(x_n)+(1-\alpha_{n}B)P_{K}(I-\lambda A)\left(\dfrac{x_{n}+x_{n+1}}{2}\right),~n\geq0,
\end{equation}
where $\{\alpha_{n}\}$ is the sequence in $(0,1)$ and satisfying the conditions (i)-(iii) of Theorem {\rm \ref{T1}}. Then the sequence $\{x_{n}\}$
converge strongly to a solution  $z$  of  {\rm VIP (\ref{s4e1})}, which is also unique solution of variational inequality {\rm (\ref{s1e4})}.
\end{theorem}

\subsection{Fredholm Integral Equation}

Consider a \verb"Fredholm integral equation" of the following form
\begin{equation}\label{s4e5}
x(t)=g(t)+\int_{0}^{t}F(t, s, x(s))ds,~~~~ t\in [0,1],
\end{equation}
where $g$ is a continuous function on $[0,1]$ and $F:[0,1]\times [0,1]\times \mathbb{R}\to\mathbb{R}$  is continuous.
Note that if $F$ satisfies the Lipschitz continuity condition, i.e.,
\begin{equation}\nonumber
\vert F(t,s,x)-F(t,s,y)\vert\leq \vert x-y\vert,~~~ \forall t,s\in [0,1],~~x,y\in \mathbb{R},
\end{equation}
then equation (\ref{s4e5}) has at least one solution in $L^2[0,1]$ (see \cite{Lions}).
Define a mapping $S:L^2[0,1]\to L^2[0,1]$ by
\begin{equation}\label{s4e6}
(Sx)(t)=g(t)+\int_{0}^{t}F(t,s,x(s))ds,~~ t\in [0,1].
\end{equation}
It is easy to observe that $S$ is nonexpansive. In fact, we have, for $x,y\in L^2[0,1]$,
\begin{eqnarray}\nonumber
\|Sx\|^2&=&\int_{0}^{1}\left \vert(Sx)(t)-(Sy)(t)\right\vert^{2} dt\\\nonumber
&=&\int_{0}^{1} \left\vert\int_{0}^{1}(F(t,s,x(s))-F(t,s,x(s)))ds\right\vert^{2}dt\\\nonumber
&\leq&\int_{0}^{1} \left\vert\int_{0}^{1}\vert x(s)-y(s)\vert ds\right\vert^{2}dt\\\nonumber
&=&\int_{0}^{1}\left\vert x(s)-y(s)\right\vert^{2}ds\\\nonumber
&=&\|x-y\|^2.
\end{eqnarray}
This means that to find the solution of integral equation (\ref{s4e5}) is reduced to finding a fixed point of the nonexpansive mapping $S$ in the Hilbert
space $L^2[0,1]$. Initiating with any function $x_0\in L^2[0,1]$. The  sequence of functions $\{x_n\}$ in $L^2[0,1]$ generated by the general viscosity
implicit midpoint iterative scheme:
\begin{equation}\label{s4e7}
x_{n+1}=\alpha_{n}\gamma Q(x_n)+(1-\alpha_{n}B)S\left(\dfrac{x_{n}+x_{n+1}}{2}\right),~n\geq 0,
\end{equation}
where $\{\alpha_{n}\}$ is the sequence in $(0,1)$ and satisfying the  conditions (i)-(iii) of Theorem {\rm \ref{T1}}. Then the sequence $\{x_n\}$
converges strongly in $L^2[0,1]$ to the solution of integral equation (\ref{s4e5}).

\subsection{Periodic solution of a nonlinear evolution equation}

Consider the following time-dependent \verb"nonlinear evolution equation" in a  Hilbert space $H$,
\begin{equation}\label{s4e8}
\dfrac{du}{dt}+A(t)u= f(t,u),~~~ t>0,
\end{equation}
where $A(t)$ is a family of closed linear operators in $H$ and $f:\mathbb {R}\times H\to H$.
The following result is the existence of periodic solutions of nonlinear evolution equation (\ref{s4e8}) due to Browder \cite{Browder}.

\begin{theorem}\label{TT1}{\rm \cite{Browder}} Suppose that  $A(t)$ and $f(t,u)$ are periodic in $t$ of period $\omega>0$ and satisfy the following assumptions:
\begin{enumerate}
\item[{(i)}] For each $t$ and each pair $u,v\in H$,
$$Re\langle f(t,u)-f(t,v),u-v\rangle\leq 0.$$
\item[{(ii)}] For each $t$ and each $u\in D(A(t))$, $Re\langle A(t)u,u\rangle\geq 0$.
\item[{(iii)}] There exists a mild solution $u$ of equation {\rm (\ref{s4e8})} on $\mathbb{R}^{+}$ for each initial value $v\in H$.
Recall that $u$ is a mild solution of {\rm (\ref{s4e8})} with the initial value $u(0)=v$ if, for each $t>0$,
\begin{equation}\nonumber
u(t)= {\cal U}(t,0)v +\int_{0}^{1}U(t,s)fs,u(s)ds,
\end{equation}
where $\{{\cal U}(t,s)\}_{t\geq s\geq 0}$ is the evolution system for the homogeneous linear system
\begin{equation}\label{s4e9}
\dfrac{du}{dt}+A(t)u=0,~~~~ (t > s).
\end{equation}
\item[{(iv)}] There exists some $R>0$ such that
$$Re \langle f(t,u),u\rangle <0,$$
for $\|u\|=R$ and all $t\in [0,\omega]$.
\end{enumerate}
Then there exists an element $v$ of $H$ with $\|v\|<R$ such that the mild solution of equation {\rm (\ref{s4e8})} with the initial condition $u(0)=v$ is periodic
of period $\omega$.
\end{theorem}

\vspace{.25cm}
Next, we  apply the general viscosity implicit midpoint rule for nonexpansive mappings to provide an implicit iterative scheme for finding a periodic solution
of (\ref{s4e8}). As a matter of fact, define a mapping $S:H\to H$ by assigning to each $v\in H$ the value u($\omega$), where $u$ is the solution of
(\ref{s4e8}) satisfying the initial condition $u(0)=v$. Namely, we define $S$ by  $Sv=u(\omega)$, where $u$ solves (\ref{s4e8}) with $u(0)=v$.

\vspace{.25cm}
We then find that $S$ is nonexpansive.   Moreover, condition (iv) of Theorem \ref{TT1} forces $S$ to map the closed ball ${\cal B}:=\{v\in H:\|v\|\leq R\}$ into
itself. Consequently, $S$ has a fixed point which we denote by $v$, and the corresponding solution $u$ of (\ref{s4e8}) with the initial condition $u(0)=v$ is
a desired periodic solution of (\ref{s4e8}) with period $\omega$. In other words, to find a periodic solution u of (\ref{s4e8}) is equivalent to finding a fixed
point of $S$. Therefore the general viscosity implicit midpoint rule is applicable to solve  (\ref{s4e8}), in which $\{x_n\}$ is generated by the
general viscosity implicit midpoint iterative scheme:
\begin{equation}\label{s4e10}
x_{n+1}=\alpha_{n}\gamma Q(x_n)+(1-\alpha_{n}B)S\left(\dfrac{x_{n}+x_{n+1}}{2}\right),
\end{equation}
where $\{\alpha_{n}\}$ is the sequence in $(0,1)$ and satisfying the  conditions (i)-(iii) of Theorem {\rm \ref{T1}}. Then the sequence $\{x_n\}$
converges  weakly to a fixed point $v$ of $S$, and the  solution of (\ref{s4e8}) with the initial value $u(0)=\omega$ is a periodic solution of (\ref{s4e8}).

\vspace{.5cm}
\noindent
{\bf Conclusion:}
The present work has been aimed to study the general viscosity implicit midpoint rule for nonexpansive mapping and proved the strong convergence theorem for
solving fixed point for a nonexpansive mapping. Theorem {\rm\ref{T1}} extends and generalize the viscosity implicit midpoint rule of Xu {\it et al.} {\rm \cite{Xu}} and the implicit midpoint rule of
Alghamdi {\it et al.} {\rm\cite{Alghamdi}} to a general viscosity implicit midpoint rule for a nonexpansive mappings, which also includes the results of
{\rm \cite{Moudafi,Marino}} as special cases.

\end{document}